\documentclass[11pt]{article}
\usepackage[utf8]{inputenc}
\usepackage[english]{babel}
\usepackage{verbatim,amsmath,amssymb,cite,epsfig,amsthm}
\usepackage{times}
\usepackage{latexsym}
\usepackage[utf8]{inputenc}
\usepackage{hyperref}
\usepackage{ulem}

\theoremstyle{definition}
\newtheorem{theorem}{Theorem}[section]
\newtheorem{definition}[theorem]{Definition}
\newtheorem{lemma}[theorem]{Lemma}
\newtheorem{proposition}[theorem]{Proposition}
\newtheorem{corollary}[theorem]{Corollary}
\newtheorem{remark}[theorem]{Remark}
\newtheorem{example}
{Example}

\numberwithin{equation}{section}

\def\bN{\mathbb N}
\def\bR{\mathbb R}

\def\eps{\varepsilon}
\def\cH{\mathcal{H}}
\def\cI{\mathcal{I}}
\def\cL{\mathcal{L}}
\def\cP{\mathcal{P}}

\def\vfi{\varphi}
\def\qed{\hfill$\Box$}
\def\di{\hbox{div}\,}

\def\sgn{\hbox{sgn}\,}
\def\lb{\langle}
\def\rb{\rangle}

\begin{document}
 \title{\bf On the Dirichlet problem for the one-dimensional ROF functional}
\author{Piotr Rybka\\ Institute of Applied Mathematics and Mechanics,
Warsaw University\\ ul.Banacha 2, 02-097 Warsaw, POLAND \\ \small
{\tt rybka@mimuw.edu.pl}\\  \small ORCiD ID \url{https://orcid.org/0000-0002-0694-8201}}


\maketitle
 \date{}

\noindent\rightline{\it in memory of Ha\"im Brezis}
 
\begin{abstract}
We provide a number of sufficient conditions for that minimizers of the one-dimensional Rudin-Osher-Fatemi  functional satisfy the Dirichlet data in the trace sense. For this purpose we use results specific for the total variation flow. We also show a number of counterexamples. 
\end{abstract}

\section{Introduction}
We study  the solvability of the minimization problem for
the Rudin-Osher-Fatemi functional (ROF functional for short) $\Phi_{f,\lambda, \vfi}$ in one dimension over  functions satisfying  the Dirichlet boundary conditions, i.e.,
\begin{equation}\label{dbc}
 \gamma u = \vfi,
\end{equation}
where $\gamma:BV(\Omega) \to L^1(\partial\Omega)$ is the trace operator.

More precisely, we define the ROF functional over  $L^2(\Omega)$ 
as follows
\begin{equation}\label{rof}
 \Phi_{f,\lambda, \vfi}(u) = \left\{ 
 \begin{array}{ll}
  \int_0^L |Du| + \frac \lambda2 \int_0^L (u - f)^2\, dx,\quad \gamma u = \vfi & u \in BV(\Omega),\\
  +\infty & 
  u \in L^2(\Omega)\setminus BV(\Omega).
 \end{array}
\right.
\end{equation}
where $\Omega = (0,L)$.
Here, $|Du|$ is the variation of the measure $Du$, $f\in L^2(\Omega)$ and $\vfi:\partial\Omega\to \bR$, i.e., $\vfi(0) = a,$  $\vfi(L) = b$.

The first issue we face is the lack of the lower semicontinuity of $ \Phi_{f,\lambda, \vfi}$. If we want to study minimizers of $\Phi_{f,\lambda,\vfi}$, then we 
have to consider the lower semicontinuous envelope, $\bar  \Phi_{f,\lambda, \vfi}$ instead of $ \Phi_{f,\lambda, \vfi}$. It is well-known, see \cite{soucek}, that we have
\begin{align}\label{rel-rof}
\bar\Phi_{f,\lambda, \vfi} (u) &= \Phi_{f,\lambda, \vfi}(u) + \int_{\partial\Omega} | \gamma u - \vfi|\, d\cH^0 \\ 
& \equiv\Phi_{f,\lambda, \vfi}(u) + | \gamma u - \vfi|(0) + | \gamma u - \vfi|(L). \nonumber
\end{align}
Now, by the direct method of the calculus of variation we can guarantee the existence of solutions to the following minimization problem,
\begin{equation}\label{r-min}
 \bar \Phi_{f,\lambda, \vfi}( u_{f,\lambda, \vfi})=\min\{ \bar\Phi_{f,\lambda, \vfi}(v): v\in BV(\Omega)\}.
\end{equation}
Moreover, the strict convexity of $\bar\Phi_{f,\lambda, \vfi}$ implies uniqueness of $u_{f,\lambda, \vfi}$. However, in general $\bar \Phi_{f,\lambda, \vfi}( u_{f,\lambda, \vfi})\le
\Phi_{f,\lambda, \vfi}( u_{f,\lambda, \vfi})$, because $u_{f,\lambda, \vfi}$ need not satisfy (\ref{dbc})
in the trace sense.



Inspired by Ha\"im Brezis Open Problem 3 in  \cite{brezis-19},
we would like to present quite general sufficient conditions for 
$$
\bar  \Phi_{f,\lambda, \vfi}(u_{f,\lambda, \vfi}) =  \Phi_{f,\lambda, \vfi}(u_{f,\lambda, \vfi}) ,
$$
where $u_{f,\lambda, \vfi}$ is a solution of (\ref{r-min}). 
Brezis  stated his question for $\lambda =1$. We find it advantageous to immerse it 
into a broader family of problems depending on the triple $D=(f,\lambda, \vfi)$. 
Our analysis depends on all pieces of the data so that we  make this dependence explicit in the definition of the ROF functional. At the same time  for the sake of simplicity of notation we will write $\bar \Phi_D$ for $\bar\Phi_{f,\lambda, \vfi} $ and $u_D$ for $u_{f,\lambda, \vfi} $. 

We will be able to answer the question of attainment of the boundary data in the trace sense in the following  cases:\\
(1) The forcing term $f\in BV(\Omega)$ satisfies the boundary conditions, $\gamma f = \vfi$. Then we can show that the minimizer $u_D$ inherits this property, see Theorem \ref{t4.1}. The converse does not hold.\\
(2) When the data $f, \lambda$ are small, i.e. $\lambda \| f\|_{L^1}\le 1$,  then $u_D$ is constant, see Theorem \ref{t4.2}. Hence, the answer to our question is affirmative when $u_D = \vfi= const$. The converse does not hold either. It is worth mentioning that  a similar phenomenon, i.e. zero solution for non-zero data, was noticed for equation
$$
\frac{d}{dx}(\cL(u_x)) = f\quad\hbox{in }(0,1),\qquad u(0) =A,\ u(1) = B,
$$
where $\cL$ is a maximal monotone graph and $f\in L^\infty(0,1)$ is small, see \cite{mucha}. Results of this sort were known earlier, see \cite[Lemma 4]{meyer}, for the ROF functional considered with the natural boundary conditions. 
\\ 
(3) The data $f\in L^2(\Omega)$ is  arbitrary, but $\lambda$ is sufficiently small. We know that there is a terminal state $u_T$ of the dynamical problem 
\begin{equation}\label{r-dyn}
 u_t \in  - \partial \Psi, \qquad u(0) =f \in L^2(\Omega),
\end{equation}
where  $\Psi = \bar\Phi_{0,0,\vfi}$, i.e. the extinction time is finite. To some extend we will treat $u_T$ as easily computable. We show that $u_T$ is the minimizer of $\bar\Phi_D$, when $\lambda$ is small. For enough regular $f$ we are able to provide sufficient conditions for that $u_T$ attains the boundary conditions in the trace sense.

Our main contribution is related to case (3). Our starting point is the observation that the minimization problem (\ref{r-min}) is in fact the time implicit semidiscretization of (\ref{r-dyn}), where $f$ is the initial state and $1/\lambda$ is the time step. Our idea is that if $\lambda$ is small i.e. the time step  $1/\lambda$ is large, then the time discretization reaches the terminal state $u_T$ of the evolution (\ref{r-dyn}). Since the terminal states satisfy $0\in \Psi(u_T)$ we deduce that $u_T$ is a minimizer of $\Psi$. In principle, we may check if $u_T$ satisfies (\ref{dbc}) in the trace sense. We also provide a sufficient condition for that, see Corollary \ref{c-HB}.

We expect that there is $0<\lambda_0$ such that $u_T$ is a miminizer of (\ref{r-min}) 
for all $\lambda\in (0, \lambda_0)$. 
Our lower bound on $\lambda_0$ is explicit, namely $\lambda_0 > \frac 12 \frac 1{\| f- u_T\|_{L^1}}$. 
We carry out this program in two stages. Initially, we assume that $f$ is nice, meaning $f\in BV(\Omega)$ it has a finite number of facets and a finite number of monotonicity changes, see Theorem \ref{t-m1}. 
Our proof that $u_T$ is a minimizer of $\bar \Phi_{f,\lambda,\vfi}$ depends on analysis of behavior of solutions to (\ref{r-dyn}) in a boundary layer. We prove a sort of a precise maximum principle. For this purpose we devote our attention to the evolution problem (\ref{r-dyn}) in Section \ref{s3}.

Once we complete the program for nice data $f$ we return to the general data, i.e. $f\in L^2(\Omega)$, this is done in Theorem \ref{t-m2}. As expected we will approximate $f$ by functions for which we can answer the original question. The desired approximation is provided by the solution to (\ref{r-dyn}) with initial condition $f$. Here, we need to recall the appropriate result for the solutions of  (\ref{r-dyn}).

Let us stress that our ability to answer the original question does not mean that the answer is positive. In addition, our approach does not give necessary conditions.



We close the paper by showing examples of data for which minimizers of (\ref{r-min}) certainly do not satisfy (\ref{dbc}) in the trace sense.
We conclude by  saying that we have a range of conditions leading to a positive answer to the original question, but none of them seems also necessary.

The paper is organized as follows. In Section \ref{s2} we describe the ROF functional and its basic properties. In Section \ref{s3} we discuss the total variation flow and the Comparison Principle. The main affirmative results are shown in Section \ref{s4}, while the negative ones are established in Section \ref{s5}.

\section{The ROF functional over $BV(0,L)$ with Dirichlet boundary condition} \label{s2}

In this section we gather the preliminary  material regarding convex analysis.

\subsection{The subdifferential of $\bar\Phi_{f,\lambda,\vfi}$}

For the sake of simplicity of notation we shall frequently write $\bar\Phi_D$,  instead of $\bar\Phi_{f,\lambda,\vfi}$, when the actual data are less important. Obviously, functional  $\bar\Phi_D$ defined in (\ref{rel-rof}) is convex. 
We recall that due to \cite{soucek} functional $\bar\Phi_D$ is also lower semicontinuous in the $L^1$ topology. Hence, the subdifferential of $\bar\Phi_D$,
$$
\partial \bar\Phi_D(u)=\{ z\in L^2(\Omega): \forall h\in L^2(\Omega)\quad 
\bar\Phi_D(u+ h) - \bar\Phi_D(u) \ge \lb z, h\rb\}
$$
is well defined. 
 
It is also well-known that  $u_D$ is a solution to the  minimization problem (\ref{r-min})
if and only if 
\begin{equation}\label{r-0ink}
 0\in \partial \bar\Phi_D. 
\end{equation}

In order to extract  properties of minimizers of $\bar\Phi_D$ we have to analyse  inclusion (\ref{r-0ink}). In fact, the characterization of the subdifferential $\partial \bar\Phi_{0,0,\vfi}$ is known for $D=(0,0,\vfi)$, see \cite{andreu1}, \cite{andreu2} and \cite{mazon}. This leads to the following statement, the details are left to the interested reader: 
\begin{proposition}\label{p2.1}
If  $\bar\Phi_{f,\lambda,\vfi}$ is defined by (\ref{rel-rof}), $\lambda\ge 0$, then 
$z\in \partial \bar\Phi_{f,\lambda,\vfi}(u)$ if and only if $z\in W^{1,\infty}(\Omega)$ and 
\begin{align}\label{r2.2}
 &\| z \|_{L^\infty} \le 1;\nonumber\\
& (z, Du) = |Du| \quad\hbox{as measures};\nonumber\\
&0= -z_x + \lambda(u- f)\equiv - \di z + \lambda(u- f);\\
&[z, \nu^\Omega]\in \sgn(\vfi - u) \quad\hbox{on }\partial\Omega,\nonumber
\end{align}
where $\nu^\Omega$ is the outer normal to the boundary of $\Omega$ and $\sgn$ is multivalued sign function, i.e. $\sgn 0 = [-1,1]$. Since our domain is one-dimensional, then the last condition above reads as
\begin{equation}\label{r2.3}
- z(0) \in \sgn (\vfi(0) - \gamma u(0)),\qquad
z(L)\in \sgn (\vfi(L) - \gamma u(L)).
\end{equation}\qed
\end{proposition}
The difficulty related with this proposition is that the absolute value appearing in the definition of the total variation $\int_\Omega |Du|$ is not differentiable. Its linear growth is another issue. 

We may approach the question of the Euler-Lagrange equation from a different direction. Namely, we may approximate $|\cdot|$ by a smooth function of the quadratic growth, e.g. $W_\eps(p) +  \frac \eps 2 p^2$, where $W_\eps$ is a mollification of the absolute value. Then, we show that the Euler-Lagrange for the regularization of $\Phi_D$, i.e.
$$
\Phi^\eps_D (u) = \int_\Omega (W_\eps(u_x) +  \frac \eps 2 |u_x|^2 + \frac\lambda2 (u- f)^2)\, dx
$$
converges to $z$. In this way we may recover (\ref{r2.2}) except for (\ref{r2.2}$_2$). This program was carried out in \cite{NaRy}.

\subsection{The continuous dependence of minimizers upon data}


After gathering the preliminary material we  present our first observation on the stability of minimizers with respect to data. On the one hand they looks promising, on the other hand it turns out to be our tool in our construction of counterexamples.

\begin{proposition}\label{pp1}
(a) Let us suppose that $f_n$, $n\in \bN$,  belong to $L^2(\Omega)$, $\lambda>0$ and $u_n$ are the  corresponding minimizers of $\bar \Phi_{f_n,\lambda, \vfi}$. If $f_n$ converge to $f_0$ in $L^2(\Omega)$, then  $u_n$ go to $u_0$ in  $L^2(\Omega)$. Moreover, 
$$
\lim_{n\to \infty} \bar \Phi_{f_n,\lambda, \vfi}(u_n) = \bar \Phi_{f_0,\lambda, \vfi}(u_0).
$$
(b) Let us suppose that $f\in L^2(\Omega)$, $\lambda>0$, $\{\vfi_n\}_{n=1}^\infty\subset \bR^2$ and $u_n$ are the  corresponding minimizers of $\bar \Phi_{f,\lambda, \vfi_n}$. If $\vfi_n$ converge to $\vfi_0$ in $\bR^2$, then $\lim_{n\to\infty} u_n=u_0$ in  $L^2(\Omega)$. Moreover, 
\begin{equation}\label{p1-r1}
\lim_{n\to \infty} \bar \Phi_{f,\lambda, \vfi_n}(u_n) = \bar \Phi_{f_0,\lambda, \vfi}(u_0).
\end{equation}
(c) Let us suppose that $f\in L^2(\Omega)$, $\{\lambda_n\}_{n=1}^\infty\subset \bR_+$ and $u_n$ are the  corresponding minimizers of $\bar \Phi_{f,\lambda_n, \vfi}$. If $\lambda_n$ converge to $\lambda_0>0$, then $u_n\to u_0$ in  $L^2(\Omega)$. Moreover, 
\begin{equation}\label{p1-r2}
\lim_{n\to \infty} \bar \Phi_{f,\lambda, \vfi_n}(u_n) = \bar \Phi_{f_0,\lambda, \vfi}(u_0).
\end{equation}
\end{proposition}
\begin{proof}
Part (a), we will first show convergence of minimizers.
Since $u_n$ is a minimizer, then $0\in \partial\Phi_{f_n,\lambda, \vfi}(u_n)$. Now, we are using the characterization of the elements of the subdifferential provided by Proposition \ref{p2.1}. 
Since  $0\in \partial\Phi_{f_n,\lambda,\vfi}(u_n)$,  then there exist $z_n\in W^{1,\infty}(\Omega)$ such that
$$
0 = -\di z_n + \lambda (u_n- f_n).
$$
We take the difference of the right-hand-side (RHS for short) above for $n$ and $k$. Subsequently, we compute the inner product of the result with $u_n- u_k$. Hence we reach,
$$
0=(\di z_n - \di z_k, u_n- u_k) + \lambda((u_n-f_n) - (u_k- f_k), u_n- u_k).
$$
Now, we use the fact that the subdifferential of the total variation is a maximal monotone operator. Thus, we reach
$(\di  z_n - \di z_k, u_n- u_k)\ge 0$. Hence, we deduce,
$$
0 \ge  \| u_n- u_k\|^2  - (f_n- f_k, u_n-u_k).
$$
At this point the Cauchy-Schwarz inequality yields
$$
\| f_n- f_k \| \ge \| u_n- u_k\|,
$$
i.e. $u_n$ is a Cauchy sequence in $L^2$ converging to $u_0\in L^2(\Omega)$.

Now, we are going to show that $u_0$ is a minimizer. Let us take any $\bar u\in BV(\Omega)$. We have to show that $\Phi_{f_0,\lambda,\vfi}(\bar u) \ge \Phi_{f_0,\lambda,\vfi}(u_0)$. Of course, we have
$$
\bar\Phi_{f_n,\lambda,\vfi}(\bar u) \ge \bar\Phi_{f_n\lambda,\vfi}(u_n).
$$
We take the limit on both sides. Then, we have
$$
\bar\Phi_{f_0,\lambda,\vfi}(\bar u) = \lim_{n\to\infty} \bar\Phi_{f_n\lambda,\vfi}(\bar u) \ge \liminf_{n\to\infty} \bar\Phi_{f_n,\lambda,\vfi}(u_n) \ge \bar\Phi_{f_0,\lambda,\vfi}(u_0),
$$
where we used the lower semicontinuity of $\bar\Phi_{0,0,\vfi}(u)\equiv \int_\Omega |Du| + \int_{\partial\Omega}|\gamma u - \vfi|$.

Now, we can show that the convergence $u_n\to u_0$ is better than just in $L^2$.
The lower semicontinuity of $\bar\Phi_{f,\lambda,\vfi}$ and the convergence  $f_n\to f$ in $L^2(\Omega)$ yield
$$
\liminf_{n\to\infty} \bar\Phi_{f_n,\lambda,\vfi}(u_n) \ge \bar\Phi_{f_0,\lambda,\vfi}(u_0).
$$
At the same time we have
$$
\bar\Phi_{f_0,\lambda,\vfi}(u_0) = \lim_{n\to\infty}\bar\Phi_{f_n,\lambda,\vfi}(u_0) \ge
\liminf_{n\to \infty}\bar\Phi_{f_n,\lambda,\vfi}(u_n) \ge \bar\Phi_{f_0,\lambda,\vfi}(u_0).
$$
Hence, all inequalities become equalities 
and our claim follows.


Part (b). We are going to show the continuous dependence of solution on the boundary data.
First of all we notice that for all $v\in BV(\Omega)$, we have 
$$
\bar\Phi_{f_n,\lambda,\vfi} (v) = \bar\Phi_{f_n,\lambda,\vfi_k} (v) + |\gamma v - \vfi_n| -|\gamma v - \vfi_k|.
$$
Now, the triangle inequality implies that $ ||\gamma v - \vfi_n| -|\gamma v - \vfi_k|| \le | \vfi_n - \vfi_k|$. Hence,
\begin{equation}\label{r3}
 \lim_{n\to\infty} \bar\Phi_{f_n,\lambda,\vfi_n}(v) = \bar\Phi_{f_n,\lambda,\vfi_0}(v) .
\end{equation}

Let us now suppose that $u_n$ is a minimizer of $\bar\Phi_{f_n,\lambda,\vfi_n}$. Since
$|Du_n|(\Omega)\le \bar\Phi_{f_n,\lambda,\vfi_n}(0)$ we deduce that there is $M>0$ such that $\| u_n\|_{BV}\le M$ for all $n\in \bN$. Thus, we may extract a subsequence (not relabeled) converging in $L^2$ to $u_0$. 
By the lower semicontinuity of $\bar\Phi_{f_n,\lambda,\vfi_0}$, we have
\begin{equation}\label{r4}
 \liminf_{n\to \infty} \bar\Phi_{f_n,\lambda,\vfi_0} (u_n) \ge \bar\Phi_{f_n,\lambda,\vfi_0} (u_0).
\end{equation}
As a result $u_0\in  BV(\Omega).$

We wish to show that $u_0$ is a minimizer of  $\bar\Phi_{f_n,\lambda,\vfi_0}$. We take $v\in BV(\Omega)$. Then we have
$$
\bar\Phi_{f_n,\lambda,\vfi_n}(v) \ge  \bar\Phi_{f_n,\lambda,\vfi_n}(u_n) =
\bar\Phi_{f_n,\lambda,\vfi_0} (u_n) + |\gamma u_n-\vfi_0| - |\gamma u_n-\vfi_n|.
$$
Since the difference of boundary terms goes to zero, we see that (\ref{r3}) and (\ref{r4}) imply the following inequalities
\begin{equation}\label{r5}
\bar\Phi_{f_n,\lambda,\vfi_0} (u_0) = \lim_{n\to\infty} \bar\Phi_{f_n,\lambda,\vfi_n}(v) \ge \liminf_{n\to \infty} \bar\Phi_{f_n,\lambda,\\vfi_n}(u_n) = \liminf_{n\to \infty} \bar\Phi_{f_n,\lambda,\vfi_0}(u_n)  \ge \bar\Phi_{f_n,\lambda,\vfi_0}(u_0). 
\end{equation}
Since $u_0$ is a unique minimizer of $\bar\Phi_{f,\lambda,\vfi_0}$, we deduce that not only a subsequence of $u_n$ converges to $u_0$ but the whole sequence does. The first part of (b) follows.

In order to show (\ref{p1-r1}) we notice that (\ref{r3}) and (\ref{r5}) imply
$$
\bar\Phi_{f_n,\lambda,\vfi_0}(u_0) =  \lim_{n\to\infty} \bar\Phi_{f_n,\lambda,\vfi_n}(u_0) \ge
\liminf_{n\to\infty} \bar\Phi_{f_n,\lambda,\vfi_n}(u_n) \ge\bar\Phi_{f_n,\lambda,\vfi_0}(u_0). 
$$
As as a result all inequalities are equalities. Formula (\ref{p1-r1}) holds.

(c) If $\lambda_n \to \lambda_0$, then the corresponding minimizers form a bounded sequence in $BV$. We may select a subsequence convergent to $u_0$ in $L^2$. From this point the argument goes like that in part (b). We deduce that $u_0$ is the unique minimizer and (\ref{p1-r2}) follows.
\end{proof}

This proposition is not sufficient to claim that the limit minimizer, $u_0$, has the right trace.

\section{The gradient flow of the total variation with the Dirichlet boundary condition}\label{s3}



We present here the information necessary for the development of 
tools related to the 
gradient flow generated by 
a lower semicontinuous convex functional $\bar \Phi_{0,0,\vfi}$. We offer two viewpoints on this question. The first one is based on the convex analysis tools.  In this approach the gradient flow
is understood as the differential inclusion (\ref{r-dyn}).

The second viewpoint is based on the approximation the differential inclusion (\ref{r-dyn}) by strictly parabolic problems, see \cite{mury} and \cite{milena}.

We first notice that the existence of the dynamical problem (\ref{r-dyn}) follows from the well-established theory of monotone operators, see \cite{brezis}. We can show:
\begin{proposition} ~ \\
(a) Let us suppose that $u_0\in D(\partial \bar\Phi_{0,0,\vfi})$. Then, there is a unique solution to (\ref{r-dyn}), i.e.
$$
u_t \in - \partial \bar \Phi_{0,0,\vfi}(u), \qquad u(0) = u_0
$$
that 
$$
u\in C([0,\infty), L^2(\Omega)), \quad \frac{du}{dt}\in L^\infty(0,\infty), \ u(0) =u_0.
$$ 
(b) Let us suppose that $u_0\in L^2$. Then, there is a unique solution to (\ref{r-dyn}). Moreover, 
$$
u\in C([0,\infty), L^2(\Omega)),\quad  t\frac{du}{dt}\in L^\infty(0,\infty, L^2(\Omega)), \quad u(0) =u_0
$$
and for all $t>0$, we have $u(t)\in  D(\partial \bar\Phi_{0,0,\vfi})$. 
In addition for all $t>0$, we have
$$
\frac{d^+}{dt} u = \partial \bar\Phi_{0,0,\vfi}^o (u),
$$
where $\frac{d^+}{dt} u $ denotes the right derivative and $\partial \bar\Phi_{0,0,\vfi}^o(u)$ is the minimal section of the subdifferential. We recall that $v= \partial \bar\Phi_{0,0,\vfi}^o(u)$ if it has the smallest $L^2$ norm of all elements of $\partial \bar\Phi_{0,0,\vfi}(u)$.
\end{proposition}
\begin{proof}
Since $\bar \Phi_{0,0,\vfi}$ is convex and lower semicontinuous on $L^2(\Omega)$, then we may apply \cite[Th\'eor\`eme 3.1]{brezis} to deduce part (a). Part (b) follows from \cite[Th\'eor\`eme 3.2]{brezis}.
\end{proof}

We remark that the minimality of the selection of $\partial \bar\Phi_{0,0,\vfi}$ means that (\ref{r-dyn}) takes the following form,
$$
u_t = z_x,
$$
where $z(\cdot, t)\in W^{1,\infty}(\Omega)$, $|z(x,t)|\le 1$ for a.e. $(x,t)\in \Omega\times (0,T)$ and $z(\cdot, t)$ is a minimizer of the following functional
\begin{equation}\label{r-msec}
\int_\Omega \zeta_x^2\, dx, \quad \zeta\in \partial\bar \Phi_{0,0,\vfi}(u).
\end{equation}

This Proposition does not say anything about attainment of the boundary conditions. For this purpose we restate the condition (\ref{r2.3}). 
Below we recall the notion introduced in \cite{mazon} and later for the one dimensional problems in \cite{milena}:  
\begin{definition}\label{d-bc}
We shall say that $u$ a minimizer of (\ref{r-min}) (resp. a  solution to (\ref{r-dyn})) satisfies the Dirichlet boundary condition 
$$
u = \vfi \qquad\hbox{on } \partial\Omega
$$
(resp. satisfies it at time $t$)
in the {\it viscosity sense} provided that 
$$
\gamma u(0) = \vfi(0) \quad \hbox{or}\quad\left\{ 
\begin{array}{ll}
 \hbox{if } \gamma u(0)>\vfi(0), & \hbox{then, } z(0) =1,\\
 \hbox{if } \gamma u(0)<\vfi(0), & \hbox{then, } z(0) =-1 
\end{array}
\right. 
$$
and
$$
\gamma u(L) = \vfi(L)  \quad \hbox{or}\quad \left\{ 
\begin{array}{ll}
 \hbox{if } \gamma u(L)>\vfi(L), & \hbox{then } z(0)=-1,\\
 \hbox{if } \gamma u(L)<\vfi(L), & \hbox{then } z(0) =1 .
\end{array}
\right. 
$$
In the case $u$ is a solutions to (\ref{r-dyn}), we replace $\gamma u(0)$ or $\gamma u(L)$ by $\gamma u(0,t)$ or $\gamma u(L,t)$, respectively for a.e. $t>0$.
\end{definition}

It is natural to ask if solutions to (\ref{r-dyn})  satisfy the Dirichlet boundary conditions in this weak sense. In order to address this question we prefer to present a  different approach to (\ref{r-dyn}) based on approximation of this eq. by strictly parabolic problems with smooth coefficients developed in \cite{mury} and \cite{milena}, because we want to use methods established there.

Since we want to use results of \cite{milena} to analyse properties of solutions to (\ref{r-dyn}), then 
we recall the language of that paper, which  was devoted to 
\begin{align}\label{r-mi}
 & u_t = \left( \frac d{dp} W(p)\right)_x,& (x,t) \in (0,L) \times (0,\infty),\nonumber\\
 & u(x,0) = u_0(x), & x \in (0,L),\\
 & u(0,t) = A, & u(L,t) = B,\nonumber
\end{align}
where  the boundary conditions were understood in the sense of Definition \ref{d-bc} and the emphasis was put on 
$W(p) = |p - 1| + |p+1|$. However, the  results of \cite{milena} are extendable to more general piecewise linear $W$. We write $\cL (p):=  \frac d{dp} W(p)$. Since $W$ is convex, then  $\cL $ is increasing with possible discontinuities, their number is finite, when $W$ is piecewise linear. Here we are mainly interested in  $W(p) = |p|$.

Existence of solutions to (\ref{r-mi}) was established by approximating  this equation by strictly parabolic problems with smooth coefficients, see \cite{milena} and \cite{mucha}. We also showed in \cite[Theorem 2.3, part (1)]{milena} that the solution we constructed is the solution provided by \cite{brezis}. Hence, by uniqueness provided by \cite[Th\'eor\`eme 3.2]{brezis} both solutions coincide. 
However, we do not expect that $u(t)$ satisfies the 
boundary data in the trace sense for a.e. $t>0$.


The object frequently used to describe properties of solutions to (\ref{r-dyn}) is a facet. We follow the definition in \cite[Definition 4.1]{milena}. It is recalled in the form suitable for $W(p) = |p|$.

\begin{definition} \label{2}{\rm Let us set $\cP$ be the set of singularities of $W$, in the present case $\cP= \{0\}$.

 (a) We shall say that a subset $F$ of the graph of a solution
to (\ref{r-mi}) is a {\it facet}, if 
$$
F \equiv F(\xi^-,\xi^+)=\{ (x,u(x)):\ u_x|_{[\xi^-,\xi^+]} \equiv p \in \cP
\}
$$
and it is maximal with this property. 

We write $u_x|_{[\xi^-,\xi^+]}$ with the understanding that the one-sided
derivatives of $u$ exist at $\xi^+$ and $\xi^-$.
Moreover, maximality means that if $[\xi^-,\xi^+]\subset J$, $J$ is an interval  and $u_x|_J\equiv p\in \cP$, then $[\xi^-,\xi^+] =J$. Interval  $[\xi^-,\xi^+]$ is called the
pre-image of facet $F$ or a faceted region of $u$. 

(b) Facet $F(\xi^-,\xi^+)$ has {\it zero curvature}, if: (i)
either $z(\cdot, t)$ 
has the same value at $\xi^-$ and $\xi^+$ or (ii) $\xi^- =0$ or $\xi^+ =L$ and $u(\cdot, t)$ satisfies the boundary condition at this point in the sense of Definition \ref{d-bc}, i.e. the facet touches the boundary of $\Omega$.} 
\end{definition}
\begin{remark}
 In \cite{milena} in the definition of the zero-curvature facet touching the boundary the condition that $u(\cdot, t)$  satisfies the boundary condition in the viscosity sense is missing.
\end{remark}

It will be important for us to discuss regularity of solutions on the overwhelming set of  time instances. We recall \cite[Definition 2.2]{milena}.

\begin{definition}
We shall say that $t>0$ is {\it typical} if
$$\int_\Omega\left|{u_t(x,t)}\right|^2dx<\infty\quad \hbox{\rm{and}}\quad \int_\Omega\left|{z_x 
(x,t)}\right|^2dx<\infty.$$
\end{definition}

We know that the gradient flow (\ref{r-dyn}) regularizes data in the sense that for all $t>0$ the solution $u(t)$ belongs to $D(\partial \bar \Phi_{0,0,\vfi})$. Actually, we may say a bit more. In \cite[Theorem 4.1]{milena} we showed:

\begin{proposition}\label{tmf}
 If $u_0\in BV(\Omega)$ and $u$ is the corresponding solution to (\ref{r-mi}), then for typical $t$, i.e.
for almost all $t>0$, the number of facets with non-zero curvature is
finite. 
\end{proposition}

We will also need a comparison principle. In order to simplify the exposition we will recall it only for (possibly discontinuous) solutions to (\ref{r-mi}). 

\begin{proposition}(\cite[Theorem 4.1]{miyory})\label{comprin}
Let $u$ and $v$ be  solutions of (\ref{r-mi}) 
in $\Omega_T=\Omega \times(0,T)$.
If $u^{*}\leq v_{*}$ on the parabolic boundary 
$\partial_{p}\Omega_T(=[0,T)\times\partial \Omega \cup \{ 0\}\times \overline\Omega)$ of $\Omega_T$, 
then $u^{*} \leq v_{*}$ in $\Omega_T$.
\end{proposition}
In this statement 
$u^{*}$ is defined by
$$
u^{*}(x,t)=\lim_{\varepsilon \downarrow 0} \ \sup \{ u(y,s) : \ |s-t|<\varepsilon, \
|x-y|<\varepsilon, \ (y,s)\in \Omega_T \}\quad\hbox{for }(x,t)\in\overline{\Omega_T}.
$$
We also set $u_{*}=(-u^{*})$.

We proved in \cite[Theorem 5.2]{milena} that the solutions we constructed in \cite{milena} are viscosity solutions in the sense of \cite{miyory}. Let us stress that the initial condition is assumed to be in $BV(\Omega)$.


In \cite{milena} we studied the extinction time for (\ref{r-mi}).
We showed there a finite extinction time, see \cite[Theorem 4.11]{milena}. Due to the similarity of (\ref{r-mi}) to (\ref{r-dyn}), 
we claim that all solutions to (\ref{r-dyn}) get extinguished in finite time.

\begin{proposition}\label{t-ext}
Let us suppose that $f\in L^2(\Omega)$. Then, there is $0\le T_{ext}<\infty$ such that for all $t>T_{ext}$ we have $u(t) = u(T_{ext}).$ 
\end{proposition}
\begin{proof} 
This fact was shown in \cite[Theorem 4.11]{milena} and \cite[Remark 4.12]{milena} for (\ref{r-mi}). The adjustments to the present case are trivial and they are omitted. 
\end{proof}

Let us write $u_T$ for $u(T_{ext})$. Since this is a terminal state, then we see that $\frac{\partial u_T}{\partial t}=0$, i.e.
\begin{equation}\label{0diff}
 0 \in \partial\bar\Phi_{0,0,\vfi}.
\end{equation}
This means that  $u_T$ is a critical point and in fact a minimizer of 
$$
\bar\Phi_{0,0,\vfi} (u_T) = \min\{ v\in BV(\Omega):\ \bar\Phi_{0,0,\vfi}(v)\}.
$$


This characterization is particularly important, because any terminal state of (\ref{r-dyn}) satisfies the inclusion (\ref{0diff}).

We need to discover properties of the terminal states. Our tools of their analysis require that we deal with initial conditions, which correspond to typical times of evolution, i.e
\begin{equation}\label{tzal}
 f\in BV(\Omega) \hbox{ has a finite number of facets with non-zero curvature, they all have positive length. }
\end{equation}
In other words $f$ has a finite number of (maximal) intervals where is it monotone. Moreover, $f$ may not strict minima nor maxima.

We begin with a simple observation regarding (\ref{tzal}).
\begin{lemma}\label{l3.9}
If $f$ satisfies (\ref{tzal}), then for all $x\in \Omega$ we have
$$
u_T(x) \in [\min \vfi, \max\vfi].
$$
\end{lemma}
\begin{proof}
We will use the Comparison Principle, Proposition \ref{comprin}, to establish our claim. We will compare $u$ with solutions to (\ref{r-dyn}) which are constant in space. 
We define $v_0 = \max\{ \max\vfi, \max f\}$, (resp.  $w_0 = \min\{ \min\vfi, \min f\}$). We are going to construct a solution $v$ (resp. $w$) to (\ref{r-dyn}) with initial condition $v_0$ (resp. $w_0$). We expect that $v$ and $w$  are constant in space. In this case,   the eq. (\ref{r-dyn}) takes the 
following form 
\begin{equation}\label{r-h}
\frac {dv}{dt} = \frac{z^v(L,t)- z^v(0,t)}{L},\quad v(0) = v_0 \qquad (\hbox{resp. }\frac {dw}{dt} = \frac{z^w(L,t)- z^w(0,t)}{L},\quad w(0) = w_0),
\end{equation}
where $z^v, z^w\in W^{1,\infty}(\Omega)$ and $z^v_x(x,t), z^w_x(x,t)\in [-1,1]$.

In order to define  numbers $z^i(L,t)$ and $z^i(0,t)$, $i=v, w$, we need another auxiliary object. For any  $g\in BV(\Omega)$ we set 
\begin{equation}\label{r-ty}
 \tilde g(x) = \left\{ 
\begin{array}{ll}
\vfi(0) & x =0,\\
g(x) & x\in \Omega,\\
\vfi(L) & x=L.
\end{array}
\right.
\end{equation}
Having this operation we introduce $\tilde v_0$ and $\tilde w_0$ and  we define $z^i(L,t)$ and $z^i(0,t)$, $i=v, w$, as follows:\\
(1) If for $t \ge 0$ function $[0,L]\ni x\mapsto \tilde v(x,t)$ is not monotone, then 
$z^v(0,t) =1$ and $z^v(L, t)=-1$.\\
(2) If  for $t \ge 0$ function $[0,L]\ni x\mapsto \tilde v(x,t)$ is increasing, then $z^v(0,t) = 1 = z^v(L,t)$.\\
(3) If  for $t \ge 0$ function $[0,L]\ni x\mapsto \tilde v(x,t)$ is decreasing, then $z^v(0,t) = -1 = z^v(L,t)$. 

In case we wish to determine these quantities for $w$ we set:\\
(1) If for $t \ge 0$ function $[0,L]\ni x\mapsto \tilde w(x,t)$ is not monotone, then 
$z^w(0,t) =-1$ and $z^w(L, t)=1$.\\
(2) If  for $t \ge 0$ function $[0,L]\ni x\mapsto \tilde w(x,t)$ is increasing, then $z^w(0,t) = 1 = z^w(L,t)$.\\
(3) If  for $t \ge 0$ function $[0,L]\ni x\mapsto \tilde w(x,t)$ is decreasing, then $z^w(0,t) = -1 = z^w(L,t)$. 

These definitions are ambiguous, when $\tilde v(\cdot, t)$ (resp.  $\tilde w(\cdot, t)$) is constant. Nonetheless, we choose a measurable $z^v$ (resp. $z^w$). However, the difference $z^v(L,t) - z^v(0,t)$ (resp. $z^w(L,t) - z^w(0,t)$) will be zero, so the right-land-side of (\ref{r-h}) is well-defined.

Once the boundary values of $z^i$, $i= v, w$ are set we can define $z^i(\cdot, t)\in W^{1,\infty}(\Omega)$ to be the linear function with the values specified at $x=0, L$ by the above  formula. Hence, $z^i(\cdot, t)$, $i=v, w$ minimizes the integral in (\ref{r-msec}) among $W^{1,2}(\Omega)$ functions with the same boundary values as $z(\cdot, t)$.

We notice that $\tilde v(t)$ is monotone if and only if 
$$
v(t)\in [\min \vfi, \max \vfi].
$$
Hence, our definition of $v$ and $z$ implies that if $\tilde v_0$ is not monotone, then $\frac{dv}{dt}(t)<0$ for $t\in (0, T^u)$, where
$$
T^u = \frac{v_0 - \max \vfi}{z^v(L,0) - z^v(0,0)} L.
$$
Moreover, $v(T^u) = \max\vfi$. If we set $v(t):=  \max\vfi$ and $z^v(\cdot,t) \equiv 1$ for $t\ge T^u$, then $v$ is 
a unique solution to (\ref{r-dyn}) with initial condition $v_0$, because $z^v_x \in \partial \Phi_{0,0,\vfi}(v(t))$. 

If $\tilde v_0$ is monotone, then $v(t) := v_0$, $z^v \equiv 1$ for all $t\ge 0$. We also set $T^u =0$.
 
The same argument applies to $w$ yielding $T_l\ge 0$, 
where 
$$
T_l = \frac{\min \vfi - w_0}{z^w(L,0) - z^w(0,0)} L.
$$ 
Moreover, $\tilde w(t)$ is monotone for $t\ge T_l$. Since  facets of $u$ move faster than single facets of $v$ or $w$ we deduce that 
$$
T_{ext} \le T^*:=\max \{T^u, T_l\}.
$$
We wish to apply the Comparison Principle, for this purpose we have to check that 
$$
w(t) = w^*(x,t) \le u_*(x,t)\quad\hbox{and} \quad u^*(x,t) \le v_*(x,t) = v(t)
$$
on the parabolic boundary of $\Omega \times (0,T)$.
 
Since $x=0$ and $x=L$ can be treated in the same way we will consider only $x=0$. Since $f$ satisfies (\ref{tzal}) the same is true about $u(\cdot, t)$ due to Lemma \ref{l310}. Thus,$u(\cdot, t)$ has only a finite number of non-zero curvature facets and all of them have positive length. Let us denote by $[\xi^-_i, \xi^+_i]$, $i=1,\cdots, N$ the preimages of these facets, where 
$$\xi^+_i< \xi^-_{i+1}, \qquad i=1,\dots N-1.
$$
We face two possibilities:\\
(i) $\xi^-_1>0$;\\
(ii) $\xi^-_1=0$.\\
In case (i) $u(\cdot, t)$ is not only monotone  over the interval $(0, \xi^-_1)$, but also $\tilde u(\cdot, t)$ is monotone  over the interval $[0, \xi^-_1)$. Indeed, the lack of monotonicity of $\tilde u(\cdot, t)$ would mean that the singleton $\{0\}$ is the preimage of a zero-length, non-zero curvature facet that is excluded by our assumption (\ref{tzal}). Thus, $z^u(x,t) = 1$ on $[0,\xi^-_1]$ and as a result $u_t(x,t) =0$ on this interval. 

At the same time, $-\frac 2L\le\frac{dv}{dt} \le 0$ and monotonicity of $\tilde u(\cdot, t)$ means that $v(t) > u^*(0,t) = \gamma u(0,t)$. Hence, 
\begin{equation}\label{r-nier}
 v(\tau) \ge \gamma u(0,\tau),\qquad\hbox{for all }t\in [t, t+\delta),\quad \delta>0.
\end{equation}

In case (ii) the facet touches the boundary, we have two  more subcases: \\
(a) $\vfi(0) = \gamma u(0,t)$,\\ 
(b) $\vfi(0) \neq \gamma u(0,t)$.
In the first case the facet has zero curvature and its velocity is zero. We are back to case (i) above when $v(t)> \gamma u(0,t)$. If we have equality $v(t)= \gamma u(0,t)$, then $[0,L)$ is the preimage of the single facet of $v$ that has zero curvature, hence $\frac{dv}{dt}=0$ for all $\tau\ge t$ and  (\ref{r-nier}) follows. 

When (b) takes place, then it is obvious that  (\ref{r-nier}) holds when $\frac{dv}{dt} u_t<0$ on $(0,\xi^-_1)$. In case  $\frac{dv}{dt} u_t>0$ the speeds satisfy $|u_t| > |\frac{dv}{dt}|$, hence  (\ref{r-nier}) follows again.

If we combine the above observations we conclude that 
$$\gamma u(0, t) = u^*(0,t) \le v(t)$$
for all $t\ge$. Hence we apply the Comparison Principle, Proposition \ref{comprin}. We conclude that for all times $t\in [0,T_{ext}]$ we have
$$
 w(t) = w^*(x,t) \le u_*(x,t)\le u(x,t)\le u^*(x,t) \le v_*(x,t) = v(t).
$$
Hence, 
$$
w^*(x,T^*) \le u(x,T^*) \le v^*(x,T^*).
$$ 
Since $T^*\ge T_{ext}$ and $u_T(x) = u(x,T^*)$ we reach our conclusion,
$$
\min \vfi \le u_T(x) \le \max \vfi.
$$

We noticed that $v(T^*)$ and $w(T^*)$ 
attain the boundary data in the trace sense at least at one endpoint of $\Omega$.
\end{proof}

\begin{lemma}\label{l310}
If $f$ satisfies (\ref{tzal}), then for all $t>0$ the number of non-zero curvature facets of positive length of $u(\cdot, t)$ is finite. 
\end{lemma}
\begin{proof}
 This is the content of \cite[Theorem 4.5]{milena}.
\end{proof}
We would like to establish more precise information about solutions to (\ref{r-dyn}) in a boundary layer.
\begin{lemma}\label{l311} For every $t\ge 0$ such that
the solution $u(\cdot, t)$ satisfies  (\ref{tzal})  there is $\eps>0$ such that:\\
(i) if $\gamma u(0, t)> \vfi(0)$, then $u_t(x, t)\le 0$ for $x\in(0,\eps)$;\\
(ii) if $\gamma u(0, t)< \vfi(0)$, then $u_t(x, t)\ge 0$ for $x\in(0,\eps)$;\\
(iii) if $\gamma u(L, t)> \vfi(0)$, then $u_t(x, t)\le 0$ for $x\in(L-\eps,L)$;\\
(iv) if $\gamma u(L, t)< \vfi(0)$, then $u_t(x, t)\ge 0$ for $x\in(L-\eps,L)$.
\end{lemma}
\begin{proof}
It suffices only to show (i). Due to our assumption there is a finite number of open intervals $I_1, \ldots, I_N$ such that $f |_{I_j}$ is strictly monotone and facets separating these intervals. We consider different types of behavior of the solution $u$ in a neighborhood of the boundary points. It suffices to consider $x=0$.

(a) We assume that $I_1 = (0, \xi_1)$ and $u$ satisfies the boundary condition in the viscosity sense, i.e. $\tilde u$ is monotone on $[0,\xi_1)$. If we combine it with the assumption  in part (i) we conclude that $\tilde u(\cdot, t)$ is monotone increasing. As a result, $z(0,t) = 1 = z(\xi_1,t)$. Hence,
\begin{equation}\label{r-l1}
 u_t(x,t) = z_x(x,t) =0 
\end{equation}
for $x\in (0, \xi_1)$.

(b) A zero-curvature facet touches the boundary at $x=0$. Since the zero-curvature facets do not move we see that (\ref{r-l1}) holds again for $x$ from the facet preimage.

(c) A facet touches the boundary, since this time its curvature is not zero we deduce that $u(\cdot, t)|_{I_1}$ is strictly decreasing,
then $z(0,t) = 1 = -z(\xi_1,t)$, where $\xi_1$ is the common endpoint of $I_1$ and facet preimage. As a result,
$$
u_t(x,t) = z_x(x,t) <0 
$$
for $x\in (0, \xi_1)$, because $z(\cdot, t)$ is a linear function on $(0,\xi_1)$. Our claim follows.
\end{proof}
Let us note that our assumption (\ref{tzal}) excludes zero-length facets with preimages $\{0\}$ or $\{L\}$.

For a fixed $t>0$ we set $l(t)$ (resp. $r(t)$)
to be the maximal number with the properties postulated by Lemma \ref{l311}. It turns out that intervals $[0, l(t)]$ and $[L-r(t), L]$ cannot shrink to a point.

\begin{lemma}\label{l312}
If $f$  satisfies (\ref{tzal}) and  $\gamma f(0) > \vfi(0)$ or  $\gamma f(0) < \vfi(0)$, (resp.   $\gamma f(L) > \vfi(L)$ or  $\gamma f(L) < \vfi(L)$),
then there is $\delta>0$ such that
$$
\delta \le \inf_{t\in(0,T_{ext})}\min\{ l(t), r(t)\}.
$$
Moreover, $u_t$ never changes sign on intervals $(0,\delta)$ and $(L-\delta, L)$.
\end{lemma}
\begin{proof} By previous lemma we have $l(t), r(t)>0$ for all $t\in (0,T_{ext})$. 
Let us suppose the contrary to our claim, i.e. $t_n\to t_0$ and $l(t_n)\to 0$. 
First  we consider
facet $F_n =\{ (x,u(x, t_n)): x\in(0,l(t_n))\}$ with a non-zero curvature. This happens when function
$$
\bar u(x, t_n) = \left\{ 
\begin{array}{ll}
\vfi(0) & x= 0,\\
u(x, t_n) & x\in (0, l(t_n) + \eta)
\end{array} \right. 
$$
for no $\eta>0$ is monotone. We showed in \cite[Proposition 4.1]{milena} and \cite{mury} that  
shrinking to a point a non-zero curvature facet is impossible.

Let us now assume that each facet $F_n(t) =\{ (x,u(x, t_n)): x\in(0,l(t_n))\}$ has  zero curvature.  This happens when there is $\eta>0$ such that function $\bar u(\cdot , t_n)$ is monotone. However, zero-curvature facets
do not move. They may collide with a neighboring facet. When this happens, $l$ will increase in a discontinuous way at the collision. Thus,  $l(t_n)\to 0$ is impossible.
\end{proof}

These observations lead to the following conclusion:

\begin{theorem}\label{zp}
If $f$ satisfies (\ref{tzal}) and $u_T$ is the terminal state of evolution, then there exists $\delta>0$ such that:\\
(i) if $\gamma f(0)> \vfi(0)$, then $f(x) \ge  u_T(x)$ for $x\in (0,\delta)$;\\
(ii) if $\gamma f(0)< \vfi(0)$, then $f(x) \le u_T(x)$ for $x\in (0,\delta)$;\\
(iii) if $\gamma f(L)> \vfi(L)$, then $ f(x)\ge u_T(x)$ for $x\in (L-\delta,L)$;\\
(iv) if $\gamma f(L)< \vfi(L)$, then $f(x) \le  u_T(x)$ for $x\in (L-\delta,L)$.
\end{theorem}
\begin{proof}
Lemma \ref{l310} guarantees that $u(t)$ satisfy (\ref{tzal}) for all $t>0$. At the same time due to Lemma \ref{l312} the time derivative $u_t$ has a constant sign on $(0,\delta)$ and $(L-\delta, L)$. Moreover, $ (\gamma f - \vfi) u_t\le 0$ on these intervals. Hence, our claim follows.
\end{proof}

We can also identify conditions sufficient for the attainment of the boundary data. It is rather clear that they are not necessary.

\begin{corollary}\label{w-wb}
 Let us suppose that $f\in BV$ satisfies  (\ref{tzal}) and $u_T$ is the terminal state of evolution. If in addition
 $$
 \gamma f(0) \not\in (\min \vfi, \max\vfi),\quad\hbox{and}\quad
 \gamma f(L) \not\in (\min \vfi, \max\vfi),
 $$
 then $\gamma u_T = \vfi$.
\end{corollary}
\begin{proof}
 We may assume that $\vfi(0)> \vfi(L)$. The case  $\vfi(0)< \vfi(L)$ is treated by the same argument and if  $\vfi(0)= \vfi(L)$, then Lemma \ref{l3.9} $u_T \equiv \vfi(0)$ and our claim holds.
 
If $v$ and $w$ are functions defined in the course of proof of Lemma \ref{l3.9}, then 
$$
v(t) \ge \gamma u(0,t) \ge \gamma u_T(0)\ge \vfi(0)> \vfi(L) \ge \gamma u_T(L) \ge \gamma u(L,t) \le w(t).
$$
Since
$$
v(T^u) = \vfi(0)> \vfi(L) = w(T_l),
$$
then our claim follows due to Lemma \ref{l3.9}.
\end{proof}

\section{Conditions for solvability of the Dirichlet problem in the trace sense}\label{s4}

Here is the first of our main observations.
\begin{theorem}\label{t4.1}
Let us suppose that $f\in BV(\Omega)$, $\gamma f = \vfi$, $\lambda$ is any positive number and $u$ is a minimizer of $\bar\Phi_{f,\lambda,\vfi}$. Then $\gamma u = \vfi$.
\end{theorem}
\begin{proof}
{\it Step 1.}
Since $u$ is a minimizer of $\Phi_{f,\lambda,\vfi}$, then $0\in \partial \Phi_{f,\lambda,\vfi}(u)$. We know the characterization of the subdifferential of $\Phi_{f,\lambda,\vfi}(u)$, i.e, then there exists $z\in W^{1,\infty}(\Omega)$ satisfying
\begin{equation}\label{r-gt1}
 0 = - z_x + \lambda(u-f), \qquad\hbox{in }\Omega,
\end{equation}
where for $x\in (0,L)$ we have
$$
z(x) = \left\{
\begin{array}{ll}
\in [-1, 1] & \hbox{if }Du = 0,\\
\sgn (Du) & \hbox{if }Du  \neq 0.\\
\end{array}\right.
$$
In particular, if $x\in \partial\Omega$, then we recall (\ref{r2.3}),
$$
- z(0) \in \sgn (\vfi(0) - \gamma u(0)),\qquad
z(L)\in \sgn (\vfi(L) - \gamma u(L)).
$$
Let us  integrate (\ref{r-gt1}) over $(0,x_n)$ where the sequence $x_n$ converges to $0$. 
$$
z(x_n) = z(0) + \lambda\int_0^{x_n}(u-f)\, ds
$$
{\it Step 2.}
We have to investigate the value of $z$ at $x=0$. the boundary of $\Omega$. We  claim that if $\gamma u(0) > \gamma f(0)$, then there is $x_0\in (0,L)$ such that
$$
\int_0^{x_0} (u(s) - f(s))\, ds >0.
$$
Let us suppose that the claim is false and for all $x\in (0,L)$ we have
\begin{equation}\label{r-gt2}
\int_0^{x_0} (u(s) - f(s))\, ds \le 0.
\end{equation}
We divide both sides of this inequality by $x_n $ and we want to pass to the limit $x\to 0$. We first notice that 
by the definition of the trace we have
$$
\lim_{x\to 0}\frac 1{x} \int_0^x v(s)\, ds 
=\gamma v(0).
$$
If we apply this observation to $v = u -f$ and take into account (\ref{r-gt2}), then we see,
$$
0\ge \lim_{x\to 0}\frac 1{x} \int_0^x (u(s) - f(s))\, ds
= \gamma(u -f)(0)>0.
$$
Thus, our claim follows.

{\it Step 3.} Let us suppose that the theorem is false and $\gamma u(0) > \gamma f(0)$. Since $z(0)\in \partial|\gamma (u -f)|$, then we infer
\begin{equation}\label{r-gt3}
z(0) =1.
\end{equation}
Since $\gamma u(0) > \gamma f(0)$, then Step 2 yields existence of $x_0\in (0,L)$ such that
$$
\int_0^{x_0} (u(s) - f(s))\, ds>0.
$$
This observation combined with (\ref{r-gt3}) leads to $z(x_0)>1$, but this is not possible. In the same manner we show that $\gamma u(0) < \gamma f(0)$ is impossible.
\end{proof}

We will state another attainment result. This time we assume that the data are small.
\begin{theorem}\label{t4.2}
Let us suppose that $f\in L^2(\Omega)$, $\lambda>0$ are such that $\lambda \|f\|_{L^1}\le 1$ and $\vfi = \bar f= \frac 1{|\Omega|} \int_\Omega f(x)\,dx$, then  $u = \bar f $ is a constant minimizer of $\bar\Phi_{f,\lambda, \vfi}$. Moreover, the boundary data are attained in the trace sense.
\end{theorem}
\begin{proof}
It suffices to show that if we set $u = \bar f$, then formula (\ref{r-gt1}) leads to a right selection of $\partial\bar\Phi_{f,\lambda, \vfi}$. Integrating  (\ref{r-gt1}) over $(0,x)$ yields
$$
z(x) = z(0) + \lambda \int_0^x(\bar f -f(y))\, dy.
$$
We have to show that there is a choice of $z(0)\in[-1,1]$ such that $|z(x)|\le 1$. Let us define $g(x) = \int_0^x(\bar f -f(y))\, dy$ for $x\in [0,L]$. We will estimate the length of the interval 
$$
[\min\{ g(x): x\in [0,L]\},\ \max\{ g(x): x\in [0,L]\}].
$$
If  the minimum  of $g$ is attained at $x_m$ and its maximum is at $x_M$, then we have
$$
g(x_M) - g(x_m) = \int_{x_m}^{x_M} (\bar f - f(y))\,dy \le
|\int_0^b f(y)\,dy | + \|f\|_{L^1} \le \frac2\lambda.
$$
We define $z(0)$ as follows. If $\lambda g(x_m), \lambda g(x_M)\in [-1,1]$, then we set $z(0) =0$. If $\lambda g(x_m)<-1$, then we set 
$z(0) = -1 -\lambda g(x_m)$. Similarly, if  $\lambda g(x_M)>1$, then 
$z(0) = 1-\lambda g(x_M)$. Obviously, $z(x)\in [-1,1]$ for all $x\in [0, L]$. Hence $u= \bar f$ is a minimizer provided that we have no constraints on $z(0)$ and $z(L)$. This happens only when $\gamma u = \bar f = \vfi$.
\end{proof}
\begin{remark}
Let us notice that the construction of $z(0)$ presented in the proof above does not work when $\vfi \neq \bar f$, because $z(0)$ must be equal $\pm 1$. 
\end{remark}

Now, we would like to put the attainment problem in the context of the dynamical problem, see (\ref{r-dyn}). The parameter $\lambda$ is the inverse of the time step. The variational problem is the time semidiscretization of the dynamics. 
The solution to the minimization problem for a given $\lambda$ is different from the solution of the dynamics at $1/\lambda$ with initial condition $f$. However, it is relatively easy to determine the final state of evolution, but with only an estimate on the extinction time $T_{ext}$. 

Our main observation is the following: for a range of $\lambda$ the terminal state is also a minimizer of (\ref{r-min}). In this way we become interested in  a slightly modified attainment problem, where we vary the parameter $\lambda$. 
Moreover,
once we have the final state, we can check if it satisfies the boundary condition in the trace sense. If it does, then we can give an affirmative answer to our problem by saying: for a range of  parameter $\lambda$ the minimizer of (\ref{r-min}) attains the boundary data in the trace sense.

If it happens that the terminal state does not satisfy the boundary data in the trace sense, then the answer to our problem is: for a large set 
of $\lambda$ the minimizer of (\ref{r-min}) does not satisfy the boundary data in the trace sense. 

The disadvantage of our method is that we 
give only an 
estimate on the range of $\lambda$'s. 

Let us notice that always the terminal state satisfies the boundary condition in the viscosity sense, otherwise $\di z \neq0$ and the system evolves.

\begin{lemma}\label{l3.4}
Let us suppose that $u_T$ is the terminal state of the evolution, i.e. $u_T = u(T_{ext})$ for eq. (\ref{r-dyn}) with initial condition $f\in BV(\Omega)$ and $f$ satisfies (\ref{tzal}). Then,\\
(a) $0 \in \partial\bar\Phi_{0,0, \vfi}(u_T)$, i.e. $u_T$ is a minimizer of the total variation,
$$
\int_\Omega |D u_T| + \int_{\partial\Omega} |\gamma u_T -\vfi|\, d\cH^0 \le\int_\Omega |D v| + \int_{\partial\Omega} |\gamma v -\vfi|\, d\cH^0 ,\qquad \hbox{for all } v\in BV(\Omega).
$$
(b) There is  $\lambda_0>0$ such that for all $\lambda\in(0, \lambda_0)$ $u_T$ is a minimizer of $\bar\Phi_{\lambda,f, \vfi}$.
\end{lemma}
\begin{proof}
Part (a). Since $u_T$ is the terminal state of evolution, this means that it satisfies
$$
0 \in \partial\Phi_{0, 0, \vfi}(u_T),
$$
because $\Phi_{0, 0, \vfi}(u) = \int_\Omega |D u_T| + \int_{\partial\Omega} |\gamma u_T -\vfi|\, d\cH^0$. However, the above inclusion means that $u_T$ is a minimizer of $\Phi_{f, 0, \vfi}$,

Part (b). Let us set $\lambda_1^{-1} = \frac12 \|f - u_T\|_{L^1}$. We claim that element $u_T$ is a minimizer of $\bar\Phi_{f,\eps, \vfi}$ for all  $\eps\in (0, \lambda_1)$. Hence, $\lambda_0\ge \lambda_1$.

We will show  directly that for all $v\in BV(\Omega)$ we have 
\begin{equation}\label{r6}
 \bar\Phi_{f,\eps, \vfi} (v) \ge \bar\Phi_{f,\eps, \vfi}(u_T).
\end{equation}

{\it Step 1.} We first consider $v= u_T + \Delta u$ such that
$$
 \Delta u = \sum_{i\in\cI} h_i \chi_{I_i},
$$
where intervals $I_i$ are pairwise disjoint. 

Our computations distinguish two cases:\\
(1) no interval $\bar I_i$ intersects $\partial\Omega$;\\
(2) there is an interval $I_{i_0}$ such that $\bar I_{i_0}\cap \partial\Omega \neq \emptyset.$

We will consider them one by one. If intervals $I_i$  defining $\Delta u$ are pairwise disjoint, then case (1) yields
$$
\int_\Omega |D v| + \int_{\partial\Omega}|v - \vfi|\,d\cH^0
= \int_\Omega |D u_T| + 2 \sum_{i\in \cI} |h_i| + \int_{\partial\Omega}|u_T - \vfi|\,d\cH^0,
$$
because the perturbation introduces jumps at the endpoints of all $I_i$. Here, we use the assumption (1). 
Now, we look at the difference $\bar\Phi_{f, \eps, \vfi}(u_T+\Delta u) - \bar\Phi_{f, \eps, \vfi}(u_T)=:\Delta \bar\Phi$, we see
$$
\Delta\bar\Phi =  2 \sum_{i\in \cI} |h_i| 
+
\frac{\eps}{2} 
\sum_{i\in \cI}
\int_{I_i} 
[(u_T + h_i 
-f)^2- (u_T -f)^2]\,dx.
$$
We notice that
$$
\int_{I_i} 
[(u_T + h_i -f)^2- (u_T -f)^2]\,dx = 2\int_{I_i} h_i(u_T -f)\,dx +  \|\Delta u\|_{L^2(I_i)}^2.
$$
Hence, by definition of $\lambda_1$ we obtain
\begin{align*}
\Delta\bar\Phi & = \sum_{i\in\cI}
|h_i|(2 +\eps\sgn(h_i)\int_{I_i}(u_T -f)\,dx + \frac\eps2 |h_i||I_i|)
\\
&\ge 
\sum_{i\in\cI}
|h_i|(2 - 2\eps \lambda_1^{-1}+ \frac \eps 2|h_i||I_i|)\\
&\ge \frac\eps2\|\Delta u\|^2_{L^2(\Omega)}.
\end{align*}
Our claim  follows in case (1). Now, we turn our attention to case (2). The argument presented above shows that is suffices to check that $\Delta\bar\Phi:=\bar\Phi_{f,\eps,\vfi}(u_T + h\chi_I)- \bar\Phi_{f,\eps,\vfi}(u_T)\ge 0$, where $0\in \bar I$. The argument when $L\in \bar I$ is the same.  We have
$$
\Delta\bar\Phi = |h| + |\gamma u_T + h - \vfi| - |\gamma u_T  - \vfi| + \eps h \int_{I}(u_T -f)\,dx + \frac \eps2 h^2|I|.
$$
If $u_T$ satisfies the boundary condition, i.e. $\gamma u_T  = \vfi$, then the argument used to establish (1) applies  again and $\Delta\bar\Phi\ge0$. If $\gamma u_T  \neq \vfi$, then we have to argue in a different way.

We recall that 
since $u_T$ is a terminal state it means that function $\tilde u_T$
is monotone, this follows from Proposition \ref{p2.1}. Otherwise $u_T$ is not a terminal state. It suffices to consider the case when $\tilde u_T$ is increasing. This means that $\gamma u_T >\vfi$ at $x=0$ is the only possibility when $\gamma u_T (0) \neq \vfi(0)$ and $\vfi > \gamma u_T$ at $x=L$ when $\gamma u_T (L) \neq \vfi(L)$. 

Let us now suppose that $h>0$. Then, $\Delta\bar\Phi $ takes on the following form,
$$
\Delta\bar\Phi = 2 h + \eps h \int_{I}(u_T -f)\,dx + \frac \eps 2 h^2|I|
$$
and by argument used in case (1), we see that $\Delta\bar\Phi\ge 0$.

Let us consider $h<0$, then
$$
\Delta\bar\Phi = |h|+ h + \eps h \int_{I}(u_T -f)\,dx + \frac \eps 2  h^2|I|=
\eps h \int_{I}(u_T -f)\,dx +\frac \eps 2  h^2|I|
.
$$
Again we consider two cases: (a) $|I|\le \delta$, (b) $|I|>\delta$, where $\delta$ is provided by Lemma \ref{l312}. 

In case (a) since $u_T$ is the terminal state, we invoke Theorem \ref{zp} to conclude
that $u_T\le f$ on $(0,\delta)$. Then, the integral on the RHS above is non-positive. Altogether this yields $ h \int_{I}(u_T -f)\,dx\ge0$. As a result,
$$
\Delta\bar\Phi \ge 0
$$
and our claim follows.

Suppose now (b) holds. Then, we replace the perturbation $h \chi_{I}$ by $(h +\eta)\chi_{(0,\delta)} + h\chi_{I'}$, where $\eta \in (0, - h_i)$ and $I' = I\setminus (0,\delta)$. This perturbation falls in the categories we considered, hence 
$$
0\le \Delta_\eta \bar \Phi := \eps (h+\eta) \int_{I}(u_T -f)\,dx + \frac \eps 2  (h+\eta)^2|I|. 
$$
Of course 
$$
0\le \lim_{\eta\to 0}  \Delta_\eta \bar \Phi = \Delta \bar \Phi.
$$
The essential observation is that the limit $\bar\Phi(u_T + \Delta_\eta u)$ as $\eta$ goes to zero equals  $\bar\Phi(u_T + \Delta_0 u)$, hence the claim follows.

{\it Step 2.} Let us suppose that $\Delta u\in BV(\Omega)$ is arbitrary. We can construct a sequence $g_n\in L^\infty(\Omega)$ converging to $\Delta u$ in $L^\infty$ and $|D g_n|(\Omega)\to |D\Delta u|$, i.e. the convergence is in the intermediate sense implying convergence of trace of $g_n$ to the trace of $\Delta$. Thus, we have
$$
\Delta\bar \Phi \ge \bar\Phi(u_T + g_n) -\bar\Phi(u_T) - \eta_n \equiv \Delta_n \bar\Phi- \eta_n,
$$
where $\eta_n\to 0$. Due to Step 1 we conclude that $ \Delta_n \bar\Phi \ge \frac\eps 2 \|g_n\|^2_{L^2(\Omega)}$. If we pass to the limit, $n\to \infty$, then we see 
$$
\Delta\bar \Phi \ge \frac\eps 2 \|\Delta u\|^2_{L^2(\Omega)}.
$$
Thus, for any $v\in BV(\Omega)$, we have (\ref{r6}).

By definition 
$$
\lambda_0 = \sup\{\lambda>0: u_T\hbox{ is a minimizer of }
\bar\Phi_{f, \mu, \vfi}\hbox{ for all }\mu\in(0,\lambda)\}.
$$
Since $\lambda_1$ satisfies the desired property, we deduce that $\lambda_0 \ge \lambda_1>0$.
\end{proof}

With this lemma at hand we can state the last of our main results. For the sake of clarity we deal separately with $f$ satisfying (\ref{tzal}) and general data $\in L^2(\Omega)$.

\begin{theorem}\label{t-m1}
Let us suppose that $f\in BV(\Omega)$ satisfies (\ref{tzal}) and $\vfi: \partial\Omega\to \bR$ are given. Then, there is $\lambda_0$ such that for all $\lambda\in (0, \lambda_0)$ the  terminal state of the evolution problem (\ref{r-dyn}),  $u_T$, is a  solution to the variational problem (\ref{r-min}). 
\end{theorem}
\begin{proof}
Existence of the terminal state is the result of the final extinction time, see Proposition \ref{t-ext}. The last claim follows from Lemma \ref{l3.4}.
\end{proof}

We could treat $u_T$ as easily computable, whose attainment of boundary conditions is no longer any issue. However, we may combine this theorem with Corollary \ref{w-wb}. Then we reach an clean partial answer to the question posed by Brezis in \cite[Open Problem 3]{brezis-19}.
\begin{corollary}\label{c-HB}
 Let us suppose that $f$ satisfies assumptions of the previous theorem. If in addition
 $$
 \gamma f(0) \not\in (\min \vfi, \max\vfi),\quad\hbox{and}\quad
 \gamma f(L) \not\in (\min \vfi, \max\vfi),
 $$
 then $\gamma u_T = \vfi$. \qed
\end{corollary}

Unfortunately, the condition stated in the corollary above is not necessary as the following example shows.

\begin{example}
Let us take $L=2$, $\vfi(0) = -1$, $\vfi(2)=1$ and 
$$
f(x) = \left\{\begin{array}{ll}
        0, & x\in (0, \frac 1/2) \cup [\frac 32, 2),\\
        -k, & x\in [\frac12, 1),\\
        k, & x\in [1, \frac 32),
       \end{array}\right.
$$
where $k\ge 1$. It is easy to check that in this case $$
u_T (x)=\left\{\begin{array}{ll}
             -1, &  x\in (0, 1),\\
             1, & x \in (1, 2).
            \end{array}\right.
$$
\end{example}

We are now turning our attention to the case of general $f\in L^2$. We offer only a result similar to that in Theorem \ref{t-m1}.

\begin{theorem}\label{t-m2}
Let us suppose that $f\in L^2(\Omega)$  and $\vfi: \partial\Omega\to \bR$ are given. Then,  there is $\lambda_0$ such that for all $\lambda\in (0, \lambda_0)$ the  terminal state of the evolution problem (\ref{r-dyn}),  $u_T$, is a  solution to the variational problem (\ref{r-min}). 
\end{theorem}
\begin{proof}
Existence of the terminal state $u_T$ is a consequence of Proposition \ref{t-ext}.

We know that a.e. $\tau>0$ is a typical time, as a result we may pick a sequence $\tau_n>0$ converging to zero such that each $u(\tau_n)$ satisfies the assumptions of the previous theorem. We conclude that the terminal state $u_T$ for the dynamics is the minimizer of $\bar \Phi_{u(\tau_n), \lambda, \vfi}$ for $\lambda\in (0, \lambda_0)$, where $\lambda_0 \ge 2/\| u_T - u(\tau_n)\|_{L^2}$. 

Now, let us take any $\lambda < 2/\| u_T -f\|_{L^2}$. Since (\ref{r-dyn}) is a semigroup of contractions, we see that 
$$
\| u(\tau_n) - u_T\|_{L^2} \le \| f - u_T\|_{L^2}.
$$
i.e. $\lambda < 2/ \| u_T -u(\tau_n)\|_{L^2}$. We may
apply Proposition \ref{pp1} to pass to the limit with $ n\to \infty$ to deduce that $u_T$ is the minimizer of $\bar \Phi_{f, \lambda, \vfi}$. 
\end{proof}



Unfortunately, when we only  know that $f\in L^2$ we cannot offer a clean condition for the attainment of the boundary conditions  by $u_T$. 

Let us stress that we showed a range of parameters $\lambda$ for which the terminal state of evolution of (\ref{r-dyn}) is a minimizer of (\ref{r-min}). The terminal states are minimizers of the relaxed total  variation, hence they are monotone functions as observed in \cite{milena}. Hence, we found a partial solution to \cite[Open Problem 3]{brezis-19}.



\section{Negative results}\label{s5}

We are going to present two types of results. The first one says that if $\lambda>0$ and $f\in L^2$ are fixed, then we can find boundary conditions which are too far apart for existence of minimizers attaining the data in the trace sense. Another result says that despite the continuous dependence of $u_D$ on $D$, the following implication is false:
\begin{quote}
 if  for a data $D$ the minimizer $u_D$ attains the boundary conditions in the trace sense and the distance between $D$ and $D'$ is small, then 
the minimizer $u_{D'}$ attains the boundary conditions in the trace sense.
\end{quote}

\subsection{Negative results large  data}

Essentially, the results established in the previous section, i.e. Theorem \ref{t-m1} and Theorem \ref{t-m2} were valid for small $\lambda$. Here, we shall show that if $\lambda>0$ is arbitrary and the boundary data $\vfi$ are large in a sense explained below, then the minimizer of $\bar\Phi_{f,\lambda,\vfi}$ does not satisfy the boundary conditions in the trace sense.

We have already noticed in Proposition \ref{p2.1} that the minimizer of $\bar\Phi_{f,\lambda,\vfi}$ satisfies the boundary condition in the {\it viscosity sense}, see Definition \ref{d-bc}. 
This means that if for $u\in BV(\Omega)$ and $\vfi:\partial\Omega\to\bR$ we define  function $\tilde u$ by formula (\ref{r-ty}),
then for each $a\in \partial\Omega$ there exists $U_a$, a neighborhood of $a$, such that $\tilde u |_{U_a}$ is a monotone function.

Here is our observation.

\begin{theorem} [non-existence for large boundary data]
Let us suppose that $f\in BV(\Omega)$ and 
$u\in BV(\Omega)$ is a minimizer of $\bar\Phi_{f, \lambda, 0}$. There exists 
$\vfi:\partial\Omega\to\bR$ such that
$$|\vfi(L) - \vfi(0)| > |Df|(\Omega) + \int_{\partial\Omega} |\gamma f| \, d\cH^0
$$ and $\tilde u$ defined in (\ref{r-ty}) satisfies the boundary conditions $\vfi$ in the viscosity sense. Then,\\
(1)  $u$ is a minimizer of $\bar\Phi_{f, \lambda, \vfi}$;\\
(2) $\gamma u\neq \vfi$;
\end{theorem}

\begin{proof} Part (1). There exists a minimizer of $\bar \Phi_{f,\lambda, 0}$. We can take any $\vfi(0)$ (resp. $\vfi(L)$) such that $z$ for $u$ and $\tilde u$ are the same. Hence $u$ is a minimizer of $\bar\Phi_{f, \lambda, \vfi}$. We notice that the set of $\vfi$ satisfying imposed conditions is open.
 
Part (2) is easy. Of course
$$
|\vfi(L) - \vfi(0)| >|Df|(\Omega) + \int_{\partial\Omega} |\gamma f| \ge |Du|(\Omega) + \int_{\partial\Omega} |\gamma u| + \frac\lambda 2\int_\Omega(u - f)^2\, dx.
$$
Thus,
$$
|\vfi(L) - \vfi(0)| > |Du|(\Omega) \ge |(\gamma u)(L) - (\gamma u)(0)|.
$$
As a result the boundary conditions cannot be attained in the trace sense.
\end{proof}

\subsection{The lack of continuous dependence}

We explain that it is unreasonable to expect that if $(f_1,\vfi_1)$ and  $(f_2,\vfi_2)$ are close and minimizer $u_{f_1, \lambda, \vfi_1}$ satisfies the boundary condition in the trace sense so does  $u_{f_2, \lambda, \vfi_2}$, which is close to $u_{f_1, \lambda, \vfi_1}$ by Proposition \ref{pp1}.

\begin{proposition} ~ \\
 (a) There exist $f_\eps, f_0\in L^2(\Omega)$, $\vfi\in \bR^2$, $\lambda>0$ such that $f_\eps\to f_0$ in $L^2$ and $\gamma u_{f_\eps, \lambda, \vfi} = \vfi$
 for all $\eps>0$, but  $\gamma u_{f_0, \lambda, \vfi} \neq \vfi$.\\
 (b) There exist $\vfi_\eps, \vfi_0\in\bR^2$, $\vfi\in \bR^2$, $\lambda>0$ such that $\vfi_\eps\to \vfi_0$ in $\bR^2$ and $\gamma u_{f, \lambda, \vfi_0} = \vfi$, but  $\gamma u_{f, \lambda, \vfi_\eps} \neq \vfi_\eps$
 for all $\eps>0$. 
\end{proposition}
\begin{proof}
 (a) We take any $\lambda>0$  any monotone decreasing function $f_0$. We choose $\vfi$ so that $\vfi(0) > \gamma f_0(0)$ and $\vfi(L) = \gamma f_0(L)$. Moreover, for any $\eps>0$ we define  
$$
f_\eps = \vfi(0) \chi_{(0,\eps)} + f_0  \chi_{[\eps, L)} .
$$
Obviously $\gamma f_\eps = \vfi$ for $\eps>0$ and  $\gamma f_0\neq \vfi$. We can also easily check that
$$
0 \in \partial\bar\Phi_{f_\eps, \lambda, \vfi}(f_\eps), \qquad \eps\ge 0.
$$
Hence, $f_\eps$ is a minimizer of $\bar\Phi_{f_\eps, \lambda, \vfi}$. By Proposition \ref{pp1} the limit of minimizers is the minimizer of (\ref{r-min}) for $\bar \Phi_{f_0,\lambda,\vfi}$. Hence, $f_0$ is a minimizer, but it does not satisfy the boundary conditions. Our claim follows.

(b) We proceed in a similar manner. Let us take any $\lambda$ and any decreasing function $f$. We set $\vfi_0 := \gamma f$ and
$$
\vfi_\eps(0) = \vfi_0(0) + \eps,\qquad \vfi_\eps(L) = \vfi_0(L).
$$
We notice that $\tilde f_\eps$ defined as in (\ref{r-ty})
is a decreasing function and 
$\gamma \tilde f_\eps \neq \vfi_\eps$. However,
$f_\eps$ is a minimizer of $\bar \Phi_{f_\eps, \lambda, \vfi_\eps}$. Our claim follows again.
\end{proof}


\end{document}